\documentclass[12pt]{article}

\usepackage[usenames]{color}
\usepackage{amssymb}
\usepackage{amsmath}
\usepackage{amsthm}
\usepackage{amsfonts}
\usepackage{amscd}
\usepackage{graphicx}

\usepackage{fullpage}
\usepackage{float}

\usepackage{psfig}
\usepackage{epsf}
\usepackage{breakurl}

\usepackage[normalem]{ulem}
\usepackage{mathrsfs}
\usepackage{pgf,tikz,pgfplots}
\usepackage{verbatim}
\pgfplotsset{compat=1.15}
\usetikzlibrary{arrows}
\usepackage{enumerate}
\usepackage{multicol}

\def\modd#1 #2{#1\ \mbox{\rm (mod}\ #2\mbox{\rm )}}

\usepackage{booktabs}
\usepackage{multirow}

\begin{document}

\begin{center}
\epsfxsize=4in
\end{center}

\theoremstyle{plain}
\newtheorem{theorem}{Theorem}
\newtheorem{corollary}[theorem]{Corollary}
\newtheorem{lemma}[theorem]{Lemma}
\newtheorem{proposition}[theorem]{Proposition}
\newtheorem{claim}[theorem]{Claim}

\theoremstyle{definition}
\newtheorem{definition}[theorem]{Definition}
\newtheorem{example}[theorem]{Example}
\newtheorem{conjecture}[theorem]{Conjecture}

\theoremstyle{remark}
\newtheorem{remark}[theorem]{Remark}

\begin{center}
\vskip 1cm{\LARGE\bf 
2-Adic Valuations of Quadratic Sequences
}
\vskip 1cm
Will Boultinghouse\\
Kentucky Wesleyan College\\
Division of Natural Sciences and Mathematics\\
3000 Frederica Street\\
Owensboro, KY 42301 \\
USA\\

\ \\
Jane Long\\
Stephen F. Austin State University\\
Department of Mathematics and Statistics\\
Nacogdoches, TX 75962-3040 \\
USA\\

\ \\

Olena Kozhushkina\\
Ursinus College\\
Department of Mathematics and Computer Science\\
Collegeville, PA 19426\\
USA \\

\ \\

Justin Trulen\\
Kentucky Wesleyan College\\
Division of Natural Sciences and Mathematics\\
3000 Frederica Street\\
Owensboro, KY 42301 \\
USA\\

\end{center}

\vskip .2 in
						
\begin{abstract} 
We determine
properties of the 2-adic valuation sequences
for general quadratic polynomials with integer coefficients 
directly from the coefficients. These properties include boundedness or unboundedness, periodicity, and valuations at terminating nodes. We completely describe the periodic sequences in the bounded case. Throughout, we frame results in terms of trees and sequences.
\end{abstract}

\section{Introduction}\label{SecIntro}
For $p$ prime and $n\in\mathbb{N}=\{0,1,2,3,\ldots\}$, the exponent of the highest power of $p$ that divides $n$ is called the \textit{$p$-adic valuation of $n$}, which we denote $\nu_{p}(n)$. The valuation of $0$ is defined to be $+\infty$. Formally, the valuation of a positive integer $n$ of the form $n=p^{k}d$, where $k\in\mathbb{N}$ and $d$ is an integer not divisible by $p$, is $\nu_{p}(n)=k$. We can find $p$-adic valuations of sequences by finding the valuation of each successive term. The present work considers 2-adic valuations of sequences generated from the natural numbers by evaluating quadratic functions of the form $f(n)=an^{2}+bn+c$ where $a,b,c\in\mathbb{Z}$ and $a\neq0$.

Information about sequences of valuations can be viewed in two different ways: in terms of sequences and in terms of trees. We let $(\nu_2(f(n)))_{n\geq 0}$ denote the sequence of 2-adic valuations for the quadratic function $f(n)$. Since every positive natural number $n$ can be written in the form $n=2^k d$, where $d$ is not divisible by 2, in many cases, we can determine the valuations of outputs of the quadratic function $f(n)$ using characteristics of the coefficients $a$, $b$, and $c$. The main results are given in Theorems~\ref{MainThm1} and~\ref{MainThm2}; one would anticipate these results can be extended to odd primes with some modifications, which will be addressed in future work.
\begin{theorem}\label{MainThm1} Let $f(n)=an^{2}+bn+c$ where $a,b,c\in\mathbb{Z}$ with $a\neq 0$ and, without loss of generality, $a,b,c$ are not all even. Then
\begin{enumerate}
    \item If $a$ and $b$ are even and $c$ is odd, then $\nu_{2}(f(n))=0$ for all $n\in\mathbb{N}$. 
    \item If $a$ is even and $b$ is odd, then $(\nu_{2}(f(n)))_{n\geq 0}$ is an unbounded sequence.
    \item If $a$ is odd and $b$ is even, then
    \begin{enumerate}
        \item if $b^{2}-4ac=0$, then $(\nu_{2}(f(n)))_{n\geq 0}$ is an unbounded sequence;
        \item if $b^{2}-4ac=4^{\ell}\Delta$ for $\ell\in\mathbb{Z}^{+}$ as large as possible and $\Delta\equiv\modd{1} {8}$, then $(\nu_{2}(f(n)))_{n\geq 0}$ is an unbounded sequence;
        \item if $b^{2}-4ac=4^{\ell}\Delta$ for $\ell\in\mathbb{Z}^{+}$ as large as possible,
        $\Delta\equiv\modd{m} {8}$, and $m\in\left\{2,3,5,6,7\right\}$, then the sequence $(\nu_{2}(f(n)))_{n\geq 0}$ is bounded and its minimal period length equals $2^{\ell}$.
    \end{enumerate}
    \item If $a$ and $b$ are odd and $c$ is even, then $(\nu_{2}(f(n)))_{n\geq 0}$ is an unbounded sequence.
    \item If $a$, $b$, and $c$ are odd, then $\nu_{2}(f(n))=0$ for all $n\in\mathbb{N}$. 
    
\end{enumerate}
\end{theorem}

Theorem~\ref{MainThm1} is proved in Section~\ref{SecInf}. Henceforward, we will refer to the \textit{minimal period length} simply as the \textit{period}. In Case 3, we use the discriminant to determine whether roots to $f(n)=0$ lie in the 2-adic numbers $\mathbb{Q}_2$ or the ring of 2-adic integers $\mathbb{Z}_2$. Corollary~\ref{InfSpecCor} takes care of Case 3(a). Even though the statement of this theorem only classifies these sequences as unbounded, the proofs of Cases 2 and 4 reveal more information about the 2-adic valuations. Theorem~\ref{MainThm1} represents a complete answer to when $\nu_{2}(f(n))_{n\geq0}$ is bounded or unbounded using only the coefficients of the quadratic polynomial. Furthermore, Theorem~\ref{MainThm1} gives an explicit period length for the bounded sequences which can be determined by the coefficients of the quadratic polynomial. In the unbounded cases we are able to describe what possible valuations will be for certain subsequences. Such statements are easier to frame in the sense of trees, which are discussed in Section~\ref{SecParityandTrees}. Theorem~\ref{MainThm2}, proved in Sections~\ref{BoundedSection} and~\ref{SecStrucTree}, completely determines all valuations in the non-trivial bounded case (3(c) of Theorem~\ref{MainThm1}).

\begin{theorem}\label{MainThm2}
Let $f(n)=an^{2}+bn+c$ where $a,b,c\in\mathbb{Z}$. If $a$ is odd and $b$ is even and $b^{2}-4ac=4^{\ell}\Delta$ for $\ell\in\mathbb{Z}^{+}$ as large as possible with $\ell\geq2$, $\Delta\equiv\modd{m} {8}$, and $m\in\left\{2,3,5,6,7\right\}$, then the sequence $(\nu_{2}(f(n)))_{n\geq 0}$ is bounded with period equal to $2^{\ell}$. Furthermore, we have the following 2-adic valuations:
\begin{displaymath}
\nu_{2}(f(n))=\begin{cases}

0,&\ \text{if}\ n\equiv\modd{ a^{-1}\left(1-\frac{b}{2}\right)} {2};\\

2(i-1),&\ \text{if}\ n\equiv\modd{ a^{-1}\left(2^{i-1}-\frac{b}{2}\right)} {2^{i}}\ \text{with}\ 2\leq i<\ell;\\

2(\ell-1),&\ \text{if}\ n\equiv\modd{ a^{-1}\left(2^{\ell-1}-\frac{b}{2}\right)} {2^{\ell}}\ \text{and}\ m=6,2;\\

2\ell-1,&\ \text{if}\ n\equiv\modd{ a^{-1}\left(2^{\ell-1}-\frac{b}{2}\right)} {2^{\ell}}\ \text{and}\ m=7,3;\\

2\ell,&\ \text{if}\ n\equiv\modd{ a^{-1}\left(2^{\ell-1}-\frac{b}{2}\right)} {2^{\ell}}\ \text{and}\ m=5;\\

2\ell-1,&\ \text{if}\ n\equiv\modd{ a^{-1}\left(2^{\ell}-\frac{b}{2}\right)} {2^{\ell}}\ \text{and}\ m=6,2;\\

2(\ell-1),&\ \text{if}\ n\equiv\modd{ a^{-1}\left(2^{\ell}-\frac{b}{2}\right)} {2^{\ell}}\ \text{and}\ m=7,5,3;\\
\end{cases}
\end{displaymath}
where $a^{-1}$ is the inverse of $\modd{a} {2^{\ell}}$.
\end{theorem}

The case $\ell=1$ is covered by Lemma~\ref{FinLem1}. In this case, the sequence is periodic with period equal to 2. Theorem~\ref{MainThm2} is proved in Proposition~\ref{FinThm} and Corollary~\ref{FinCor}. Both of these results are an extension of the work by Byrnes et al.~\cite{Byrnes}, which only considered quadratics of the form $f(n)=an^{2}+c$. The work of Medina et al.~\cite{Medina} details conditions under which these sequences are bounded or unbounded for general primes but we extend these results for $p=2$ by providing the exact conditions on the coefficients of quadratic equations. Furthermore, we provide a closed form giving the exact valuation for the bounded sequences relying only on the coefficients of the quadratic function. Boundedness of $p$-adic valuations of polynomial sequences is also discussed in Bell's work \cite{Bell}.

\section{Parity and trees}\label{SecParityandTrees}

Consider a quadratic function of the form $f(n)=an^2+bn+c,$ where $a,b$ and $c$ are integers and $a$ is nonzero. To prove the results stated in Theorems~\ref{MainThm1} and~\ref{MainThm2}, we consider the eight possible cases based on the parity of the coefficients $a$, $b$, and $c$. In the case where $a,b,$ and $c$ are all even, there exists an $i\in\mathbb{N}$ such that $2^{i}$ divides $a,b$, and $c$ but $2^{i+1}$ does not. Then $f(n)=2^{i}(a_{0}n^{2}+b_{0}n+c_{0})$ and 
it follows that $\nu_{2}(f(n))=i+\nu_{2}(a_{0}n^{2}+b_{0}n+c_{0})$. Hence, this case can be reduced to one of the other seven cases. So we assume, unless stated otherwise, that $a$, $b$, and $c$ are not all even.

Two more cases of Theorem~\ref{MainThm1} are trivial (Case 1 where $a,b$ are even, and Case 5, where $a,b,$ and $c$ are odd), since $\nu_{2}(f(n))=0$ for all $n\in\mathbb{N}$. For the remaining five cases, we classify the behavior using trees. In the case that $a$ is odd and $b$ is even we show, with the help of the discriminant, that $f(n)=0$ has a root in $\mathbb{Q}_2$. We must take some care since some quadratics may not have a zero in $\mathbb{Q}_2$. 

As discussed in Section~\ref{SecIntro}, we can present information about the sequence of valuations using a tree. We begin the construction of the tree with the top node representing the valuation of the quadratic $f(n)$ evaluated at any natural number $n$. If the $2$-adic valuation is constant for every $n$ in this node, then we stop the construction, as $\nu_2(n)$ is completely determined for the sequence. If $\nu_2(n)$ is not constant, this node splits into two branches, where one branch represents all numbers of the form $n=2q$ and the other branch represents all numbers of the form $n=2q+1$, where in both cases $q\in\mathbb{N}$. We then repeat this step as necessary to create the tree. The nodes correspond to the sets  $\{2^{i}q+r_{i-1}|q\in\mathbb{N}\}$ where
\begin{equation}\label{2adicRemainder}
    r_{i-1}=\sum_{k=0}^{i-1}\alpha_{k}2^{k},
\end{equation}
for fixed coefficients $\alpha_{k}\in\{0,1\}$. This process does not always terminate. If it terminates, we say that the tree is \textit{finite}; otherwise, the tree is \textit{infinite}. We say a node is \textit{non-terminating} if $(\nu_{2}(f(n)))_{n\geq0}$ is non-constant for every $n$ in that equivalence class. We say a node is \textit{terminating} if $(\nu_{2}(f(n)))_{n\geq 0}$ is constant for every $n$ in that equivalence class. In practice, we label the node with this constant valuation.

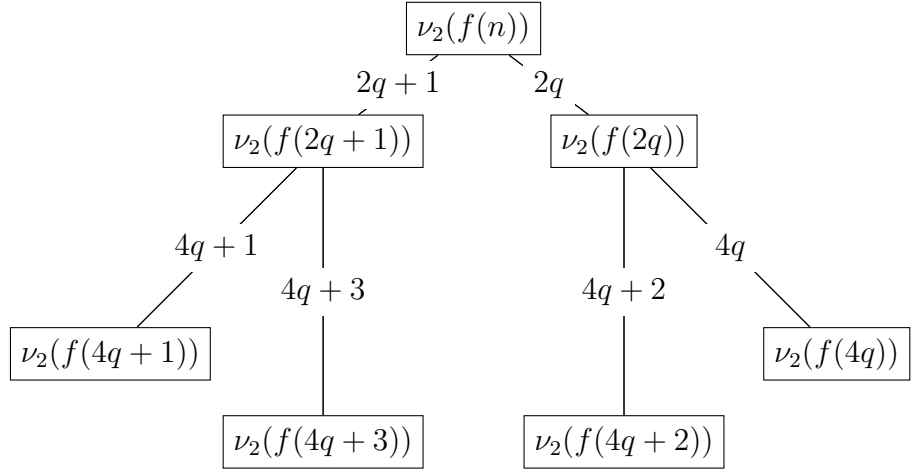
\begin{figure}
\begin{center}
\begin{tikzpicture}[level 1/.style={sibling distance=4cm},
level 2/.style={sibling distance=6cm}]
\node at ([shift={(0cm,0cm)}]current page.north)[draw](x){$\nu_{2}(f(n))$}
child{node[draw] (a){$\nu_{2}(f(2q+1))$}
    child[grow=south,level distance=4cm] {node[draw] (c){$\nu_{2}(f(4q+3))$}}
    child[grow=south west,level distance=4cm] {node[draw] (d){$\nu_{2}(f(4q+1))$}}}
child {node[draw] (b){$\nu_{2}(f(2q))$}
    child[grow=south,level distance=4cm] {node[draw] (e){$\nu_{2}(f(4q+2))$}}
    child[grow=south east,level distance=4cm] {node[draw] (f){$\nu_{2}(f(4q))$}}}
;
\path (x)  edge node[fill=white] [draw=none, midway] {$2q+1$} (a);
\path (x) edge node[fill=white][draw=none, midway]{$2q$} (b);
\path (a)  edge node[fill=white] [draw=none, midway] {$4q+3$} (c);
\path (a) edge node[fill=white][draw=none, midway]{$4q+1$} (d);
\path (b)  edge node[fill=white] [draw=none, midway] {$4q+2$} (e);
\path (b) edge node[fill=white][draw=none, midway]{$4q$} (f);

\end{tikzpicture}
\end{center}
\caption{Levels 0, 1, and 2 of a tree.}
\end{figure}

For each of the remaining five nontrivial cases on the parity of the coefficients $a$, $b$ and $c$, either $(\nu_{2}(f(n)))_{n\geq0}$ produces a finite tree or an infinite tree. We say a finite tree has \textit{$\ell$ levels} if there exists $\ell\in\mathbb{Z}^{+}$ such that for all $r_{\ell-1}\in\{0,1,2,\ldots,2^{\ell}-1\}$ we have $(\nu_{2}(f(2^{\ell}q+r_{\ell-1})))_{q\geq 0}$ constant for all $q\in\mathbb{N}$, and $\ell$ is the smallest possible value. Every node at level $\ell$ in a finite tree has a constant valuation, which depends on $r_{\ell-1}$.

Each node of a tree represents a subsequence of the sequence of 2-adic valuations. A finite tree of $\ell$ levels represents a sequence with period equal to $2^{\ell}$.

In the literature, these finite trees are also called \textit{finite automata.} The sequences generated via the 2-adic valuation are called \textit{2-automatic sequences} and, in particular, the sequences $f(2^{i}q+r)$ are known as the \textit{2-kernel sequences.} See Allouche and Shallit's book~\cite{AlloucheShallit} and Bell's paper~\cite{Bell} for more details.

\subsection{2-adic numbers and selected lemmas}\label{SecLemmas}

First, we state several well-known lemmas. The first is a well-established fact about the $p$-adic valuation, which can also be defined on the set $\mathbb{Q}$ and extends to $\mathbb{Q}_2$; see Lemma 3.3.2 in~\cite{Gouvea}.

\begin{lemma}\label{SuppLemm1}
Let $x,y\in\mathbb{Q}$, then $\nu_{p}(xy)=\nu_{p}(x)+\nu_{p}(y)$.
\end{lemma}

An element $n$ in $\mathbb{Q}_2$ can be represented in the form 
\begin{equation}\label{SuppLemma3.5}
    n=\sum_{i=k}^{\infty}\alpha_{i}2^{i}
\end{equation}
where $k=-\nu_{2}(n)$ and $\alpha_{i}\in\left\{0,1\right\}$ for all $i$; it is well-known that this representation is unique. 

Lemma~\ref{SuppLemm1} and the construction of $\mathbb{Q}_{2}$ are well-known~\cite{Gouvea}. Medina et al.~\cite{Medina} provide a useful characterization of the sequence of 2-adic valuations of a polynomial. Before we state the result, we recall the following characterization of the ring of 2-adic integers
\begin{equation*}
\mathbb{Z}_{2}=\left\{n\in\mathbb{Q}_{2}:n=\sum_{i=0}^{\infty}\alpha_{i}2^{i}\ \text{where}\ \alpha_{i}\in\left\{0,1\right\}\right\}.
\end{equation*}

\begin{lemma}\label{SuppLemm6}
(\cite{Medina}, Theorem 2.1) Let $f(n)\in\mathbb{Z}[n]$ be a polynomial that is irreducible over $\mathbb{Z}$. Then $(\nu_{2}(f(n)))_{n\geq0}$ is either periodic or unbounded. Moreover, $(\nu_{2}(f(n)))_{n\geq0}$ is periodic if and only if $f(n)$ has no zeros in $\mathbb{Z}_{2}$. In the periodic case, the minimal period length is a power of $2$.
\end{lemma}

We assume that the quadratic $f(n)$ is irreducible because, if not, by Lemma~\ref{SuppLemm1}, $$\nu_p(f(n))=\nu_p(g(n)\cdot h(n))=\nu_p(g(n))+\nu_p(h(n)),$$ where $g(n)$ and $h(n)$ are irreducible.

Therefore, to determine whether $(\nu_{2}(f(n)))_{n\geq0}$ is periodic or unbounded, it suffices to determine if $f(n)$ has zeros in $\mathbb{Q}_2$ and then determine whether the zeros are also in $\mathbb{Z}_{2}$.
The following lemmas will be used in Section~\ref{SecInf} to identify when the square root of a number is in $\mathbb{Z}_{2}$. The version of Hensel's lemma stated below determines when a polynomial in $\mathbb{Z}_{2}[x]$ has zeros in $\mathbb{Z}_{2}$. Lemma~\ref{roots}, which follows from Lemma~\ref{Hensel}, specifically determines whether the polynomial $f(x)=x^{2}-a$ has zeros in $\mathbb{Z}_{2}$.

\begin{lemma}\label{Hensel}
    (Hensel's lemma,~\cite[Sec.~6.4]{Robert}) Assume that $P\in\mathbb{Z}_{2}[x]$ and $x_{0}\in\mathbb{Z}_{2}$ satisfies \begin{equation*}
        P(x_{0})\equiv\modd{0} {2^{n}}
    \end{equation*}
    If $\phi=\nu_{2}(P'(x_{0}))<n/2$, then there exists a unique zero $\xi$ of $P$ in $\mathbb{Z}_{2}$ such that
    \begin{equation*}
        \xi\equiv\modd{x_{0}} {p^{n-\phi}}\ \text{and}\ \nu_{2}(P'(\xi))=\nu_{2}(P'(x_{0}))
    \end{equation*}
\end{lemma}
\begin{lemma}\label{roots}
(\cite[Sec.~6.6]{Robert}) The function $f(x)=x^{2}-a$ has a zero in $\mathbb{Z}_{2}^{\times}$, the set of invertible elements of $\mathbb{Z}_{2}$, if and only if $a\equiv\modd{1} {8}$.
\end{lemma}

\section{Proof of Theorem~\ref{MainThm1}: unbounded cases and infinite trees}\label{SecInf}

We now prove Theorem~\ref{MainThm1}. The main idea is to describe the roots to $f(n)=0$ in $\mathbb{Q}_{2}$ simply using the quadratic formula, the parity of the coefficients, and the lemmas presented in Section~\ref{SecLemmas}. Moreover, according to Lemma~\ref{SuppLemm6}, if a zero exists in $\mathbb{Z}_{2}$, it manifests as an infinite branch in the tree. We begin with Cases 2 and 4. 

To this end, note that in Case 2, we can write $a=2r$ and $b=2k+1$ for some $r,k\in\mathbb{Z}$. Then $an^{2}+bn+c=0$ has roots of the form 
\begin{equation}\label{rootform1}
x=\frac{-2k-1\pm\sqrt{1-8(rc-\beta)}}{4r},
\end{equation} where $\beta=(k^{2}+k)/2$. Set $j=rc-\beta$.

Also, in Case 4, we can write $a=2r+1$, $b=2k+1$, and $c=2p$. Then $an^{2}+bn+c=0$ has roots of the form
\begin{equation}\label{rootform2}
x=\frac{-2k-1\pm\sqrt{1-8((2r+1)p-\beta)}}{2(2r+1)},
\end{equation} where $\beta=(k^2+k)/2$. Set $j=(2r+1)p-\beta$. Observe that in either case the roots contain $\sqrt{1-8j}$ where $j\in\mathbb{Z}$. Since $\sqrt{1-8j}$ is a zero of the function $g(x)=x^{2}-(1-8j)$, by Lemma~\ref{roots} the zero is in $\mathbb{Z}_{2}$.

Notice that both roots \eqref{rootform1} and \eqref{rootform2} have an even denominator. We still need to check if these roots are in $\mathbb{Q}_{2}$ or $\mathbb{Z}_{2}$.  Therefore, in light of Lemma~\ref{SuppLemm6}, Case 2 (Proposition~\ref{InfCase1}) and Case 4 (Proposition~\ref{InfCase2}) are proved by an inductive argument on the behavior of the tree. It turns out that, in Case 2, $f(n)$ has exactly one zero in $\mathbb{Z}_2$ and in Case 4, $f(n)$ has two zeros in $\mathbb{Z}_2$. See Figure~\ref{fig:Ex8} in the Appendix for an example of a tree with one infinite branch and Figure~\ref{fig:Ex7} for an example of a tree with two infinite branches.

\begin{proposition}\label{InfCase1}
If $a$ is even and $b$ is odd, then the 2-adic valuation tree of $f(n)=an^2+bn+c$ has exactly one infinite branch. Furthermore, the valuation of the terminating node at the $i^{th}$ level is $i-1$.
\end{proposition}

\begin{proof}
Note that this Proposition corresponds to Case 2 of Theorem~\ref{MainThm1}. Substituting $a=2r$ and $b=2k+1$ for some $r,k\in\mathbb{Z}$, we get $an^2+bn+c=2(rn^2+kn)+n+c$. Now suppose that $c$ is even. If $n$ is even, then $2(rn^{2}+kn)+n+c$ is divisible by 2 and so $\nu_{2}(f(2n))\geq1$. If $n$ is odd, then $2(rn^{2}+kn)+n+c$ is not divisible by 2 and so $\nu_{2}(f(2n+1))=0$. An analogous argument shows that, for $c$ odd, $\nu_{2}(f(2n))=0$ and $\nu_{2}(f(2n+1))\geq1$. Therefore, the conclusion of the proposition is valid at the initial step.

Now, arguing inductively, suppose that $n=2^{i}q+r_{i-1}$ is the non-terminating node, that is $\nu_{2}(f(n))\geq i$. So $f(n)\equiv\modd{0} {2^{i}}$ or $a(2^{i}q+r_{i-1})^{2}+b(2^{i}q+r_{i-1})+c=2^{i}\beta$ where $\beta\in\mathbb{Z}$. Consider $f(n)$ evaluated at the next level:
$$a(2^{i+1}q+r_{i-1})^{2}+b(2^{i+1}q+r_{i-1})+c\equiv ar_{i-1}^{2}+br_{i-1}+c
\equiv \modd{2^{i}\beta} {2^{i+1}},$$ and
\begin{align*}
a(2^{i+1}q+2^{i}+r_{i-1})^{2}  + b(2^{i+1}q+2^{i}+r_{i-1})+c 
&\equiv ar_{i-1}^{2}+2^{i}b+br_{i-1}+c \\
& \equiv 2^{i}\beta+2^{i}b \equiv \modd{2^{i}(\beta+b)} {2^{i+1}}. 
\end{align*}
Since $b$ is odd it follows that the valuation of one node is $i$ and the other is greater than $i$ depending on if $\beta$ is odd or even. Therefore one node terminates and the other is non-terminating.
\end{proof}

\begin{proposition}\label{InfCase2}
If $a$ and $b$ are odd, and $c$ is even, then the 2-adic valuation tree of $f(n)=an^2+bn+c$ has two infinite branches. Furthermore, the valuation of the terminating node at the $i^{th}$ level is $i$.
\end{proposition}

\begin{proof}
This proposition addresses Case 4 of Theorem~\ref{MainThm1}. Write $a=2r+1$, $b=2k+1$, and $c=2p$ for some integers $r,k,$ and $p$. First note that both $a(2q)^{2}+b(2q)+c$ and $a(2q+1)^{2}+b(2q+1)+c$ are congruent to $\modd{0} {2}$. We now verify that the proposition holds at the initial step.

In the $2q$ case, we check $4q$ and $4q+2$. Note that
$$a(4q)^{2}+b(4q)+c\equiv \modd{c} {4}$$
and
$$a(4q+2)^{2}+b(4q+2)+c\equiv\modd{ 2b+c} {4}.$$ 
If $c\equiv\modd{0} {4}$, then $2b+c\not\equiv\modd{0} {4}$. If $c\not\equiv\modd{0} {4}$ then $c=2p$ with $p$ odd and $2b+c=2(b+p)\equiv\modd{0} {4}$. That is, either $$\nu_{2}(f(4q))\geq2 \text{ and }\nu_{2}(f(4q+2))=1,\text{ or}$$ $$\nu_{2}(f(4q))=1 \text{ and }\nu_{2}(f(4q+2))\geq2.$$ For the $2q+1$ case, we check $4q+1$ and $4q+3$. 
Note that $$a(4q+1)^{2}+b(4q+1)+c\equiv\modd{a+b+c} {4}$$
and $$a(4q+3)^{2}+b(4q+3)+c\equiv\modd{a+3b+c} {4}.$$ 
By hypothesis, $a+b+c=2(r+k+p)$ and $a+3b+c=2(r+3k+p+2)$. But note that $r+3k+p+2=(r+k+p+1)+(2k+1)$. Now it is clear that $r+3k+p+2$ is even (odd) if and only if $r+k+p+1$ is odd (even). Again, either $$\nu_{2}(f(4q+1))\geq2\text{ and }\nu_{2}(f(4q+3))=1\text{, or}$$ $$\nu_{2}(f(4q+1))=1\text{ and } \nu_{2}(f(4q+3))\geq2.$$

For the inductive step, now suppose that 
$n=2^{i}q+r_{i-1}$ and $n=2^{i}q+r_{i-1}^{*}$ are the non-terminating nodes where $r_{i-1}=\sum_{k=1}^{i-1}\alpha_{k}2^{k}+1$ (the odd side branch) and $r_{i-1}^{*}=\sum_{k=1}^{i-1}\alpha_{k}2^{k}$ (the even side branch) where $\alpha_{k}\in\left\{0, 1\right\}$. The fact that these branches are non-terminating follows from the same argument as in the proof of Proposition~\ref{InfCase1}.
\end{proof}

We now consider Case 3(b) of Theorem~\ref{MainThm1}.

\begin{proposition}\label{InfSpecThm}
Let $a$ be odd, $b$ be even and $b^{2}-4ac=4^{\ell}\Delta$ for some $\ell\in\mathbb{Z}^+$ as large as possible and $\Delta\equiv\modd{1} {8}$, then the 2-adic valuation tree of $f(n)=an^2+bn+c$ has two infinite branches.
\end{proposition}

\begin{proof}
Let $a$ be odd and $b=2k$ for some $k\in\mathbb{Z}$. Fix $\ell\in\mathbb{Z}^+$. Then $an^{2}+bn+c=0$ has roots of the form $x=\frac{-k\pm\sqrt{k^{2}-ac}}{a}$. By the hypothesis $4k^{2}-ac=2^{2\ell}\Delta$ where $\Delta\equiv\modd{1} {8}$. 

If $\Delta<0$ then we can naturally write $\Delta=1-8j$ where $j\in\left\{1,2,3,\ldots\right\}$.

If $\Delta>0$, then we can write $\Delta=1+8j=1-8(-j)$ where $j\in\mathbb{N}$.

Thus in either case $\Delta=1-8j$ where $j\in\mathbb{Z}$. Then 
it follows that $\sqrt{4k^{2}-4ac}=2^{\ell}\sqrt{1-8j}$. By Lemma~\ref{roots}, $\sqrt{1-8j}$ is in $\mathbb{Z}_{2}$. Furthermore, since the denominator of $x$ is odd this also guarantees that $x\in\mathbb{Z}_{2}$. Therefore, there are two infinite branches, one corresponding to each root. 
\end{proof}

\begin{corollary}\label{InfSpecCor}
Under the conditions of Proposition~\ref{InfSpecThm}, if $b^{2}-4ac=0$, the tree has one infinite branch.
\end{corollary}

\begin{proof}
In this case (3(a) of Theorem~\ref{MainThm1}) roots take the form $x=-\frac{b}{2a}$. Since $b=2k$, then $x=-\frac{k}{a}$ which has 2-adic form $x=\sum_{i=0}^{\infty}\alpha_{i}2^{i}$ where $\alpha_{i}$ is either 0 or 1. This guarantees that the one branch is infinite.
\end{proof}

\begin{remark}
Note the connection between subsequences of $(\nu_2(f(n)))_{n\geq 0}$ and the infinite branches of a tree.
Proposition~\ref{InfCase1} asserts that for all $i\in\mathbb{Z}^{+}$ there exists exactly one subsequence  of the form $n=2^{i}q+r_{i-1}$ such that $\nu_{2}(f(n))\geq i$ and exactly one subsequence of the form
$n=2^{i}q+r_{i-1}^{*}$ with $\nu_{2}(f(n))=i-1$. Similarly, Proposition~\ref{InfCase2} asserts that for all $i\in\mathbb{Z}^{+}$ there are exactly two subsequences corresponding to $n=2^{i}q+r_{i-1}$ such that $\nu_{2}(f(n))\geq i+1$ and exactly two subsequences of the form $n=2^{i}q+r_{i-1}^{*}$ with $\nu_{2}(f(n))=i$. For $r_{i-1}$ and $r_{i-1}^{*}$, the representations presented in equation (\ref{2adicRemainder}) of Section~\ref{SecParityandTrees} equate the coefficients $\alpha_{k}$ and $\alpha_{k}^{*}$ for all $0\leq k\leq i-2$, and
meanwhile $\alpha_{i-1}^{*} \equiv \modd{\alpha_{i-1}+1} {2}$.

As for the cases of Proposition~\ref{InfSpecThm} and Corollary~\ref{InfSpecCor}, we can apply Lemma~\ref{SuppLemm6} to conclude that these sequences are unbounded. Much like Propositions~\ref{InfCase1} and~\ref{InfCase2}, we can say that the results of Proposition~\ref{InfSpecThm} yield that for all $i\in\mathbb{N}$ there are exactly two subsequences of the form $n=2^{i}q+r_{i-1}$, where $(\nu_{2}(f(n)))_{n\geq 0}$ is not constant, while Corollary~\ref{InfSpecCor} asserts there is exactly one such subsequence.
\end{remark}

\section{Bounded cases and finite trees}\label{BoundedSection}

In this section, we prove Case 3(c) of Theorem~\ref{MainThm1} and the first part of Theorem~\ref{MainThm2}. The coefficients of these quadratics satisfy the following: $a$ is odd and $b$ is even, and $b^{2}-4ac=4^{\ell}\Delta$, where $\ell\in\mathbb{Z}^+$ is as large as possible, $\Delta\equiv\modd{m} {8}$, and $m\in\{2,3,5,6,7\}$. Their trees are finite with $\ell$ levels. We can again apply the reasoning of the proof of Proposition \ref{InfSpecThm}. 

If $\Delta<0$ we can naturally write $\Delta=m-8j$ where $j\in\mathbb{N}$ and if $\Delta>0$ then we write $\Delta=m+8j=m-8(-j)$ where $j\in\mathbb{N}$ or $j=0$. Henceforth, we will write $\Delta=m-8j$ where $j\in\mathbb{Z}$. Again, by Lemma~\ref{roots}  functions of the form $g(x)=x^{2}-(m-8j)$ do not have a zero in $\mathbb{Z}_{2}$. By Lemma~\ref{SuppLemm6}, the corresponding valuation sequences are periodic. Figures~\ref{fig:Ex5} and~\ref{fig:Ex6} in the Appendix illustrate examples of finite trees arising from functions $f_3(n)=15n^2+1142n+25559$
and $f_4(n)=5n^2+106n+1125$. 

We should take a moment to note why we only need to consider these five values of $m$. First note that in Cases 3(b) and 3(c) of Theorem~\ref{MainThm1}, where $a$ is odd and $b$ is even, we have the condition that $\ell$ is as large as possible. This corresponds to factoring out as many powers of 4 as possible, ruling out the possibilities $m\in\{0,4\}$. Now if $m=1$ (Case 3(b), covered in 
Section~\ref{SecInf}), an infinite tree is created. This leaves the cases $m\in\{2,3,5,6,7\}$. As discussed above, the zeros of these quadratic functions are not elements of $\mathbb{Q}_{2}$; therefore, their trees must be finite. The proofs of the next two propositions follow the proofs of Propositions~\ref{InfCase1} and~\ref{InfCase2}.

\begin{proposition}\label{FinThm}
If $a$ is odd and $b$ is even, and $b^{2}-4ac=4^{\ell}\Delta$ where $\ell\in\mathbb{Z}^+$ is as large as possible, $\Delta\equiv\modd{m} {8}$, and $m\in\{2,3,5,6,7\}$, then the 2-adic valuation tree of $f(n)$ is finite with $\ell$ levels.
\end{proposition}

The proof of this proposition is broken down into Lemmas~\ref{FinLem1},~\ref{FinLem2}, and~\ref{FinLem3}. Unless stated otherwise, let $b=2k$ for some $k\in\mathbb{Z}$. Lemma~\ref{FinLem1} covers the case $\ell=1$, in which the 2-adic valuation tree has exactly one level. Lemmas~\ref{FinLem2} and~\ref{FinLem3} describe valuations for finite trees with more than one level; Lemma~\ref{FinLem3} describes the valuation at the final level and Lemma~\ref{FinLem2} describes the other levels. Under the assumptions of Proposition~\ref{FinThm}, with $a$ odd and $b$ even, we complete the square and use properties of the $p$-adic valuation to obtain $\nu_{2}(an^{2}+bn+c)=\nu_{2}((an+k)^{2}-k^{2}+ac)$.
\begin{lemma}\label{FinLem1}
Let $\ell=1$, i.e., $b^{2}-4ac=4\Delta$, $\Delta\equiv\modd{m} {8}$, and $m\in\{2,3,5,6,7\}$. If $m\in\{2,7\}$ and $b\equiv\modd{0} {4}$ or if $m\in\{3,6\}$ and $b\equiv\modd{2} {4}$, then
\begin{displaymath}
\nu_{2}(an^{2}+bn+c)=\begin{cases}
0,&\ \text{if}\ n\ \text{even};\\
1,&\ \text{if}\ n\ \text{odd}.
\end{cases}
\end{displaymath}
If $m\in\left\{3,6\right\}$ and $b\equiv\modd{0} {4}$ or if $m\in\left\{2,7\right\}$ and $b\equiv\modd{2} {4}$, then
\begin{displaymath}
\nu_{2}(an^{2}+bn+c)=\begin{cases}
1,&\ \text{if}\ n\ \text{even};\\
0,&\ \text{if}\ n\ \text{odd}.
\end{cases}
\end{displaymath}
If $m=5$ and $b\equiv\modd{0} {4}$, then
\begin{displaymath}
\nu_{2}(an^{2}+bn+c)=\begin{cases}
0,&\ \text{if}\ n\ \text{even};\\
2,&\ \text{if}\ n\ \text{odd}.
\end{cases}
\end{displaymath}
If $m=5$ and $b\equiv\modd{2} {4}$, then
\begin{displaymath}
\nu_{2}(an^{2}+bn+c)=\begin{cases}
2,&\ \text{if}\ n\ \text{even};\\
0,&\ \text{if}\ n\ \text{odd}.
\end{cases}
\end{displaymath}
\end{lemma}

\begin{proof}
Using the convention that $\Delta=m-8j$ where $j\in\mathbb{Z}$ and $m\in\left\{2,3,5,6,7\right\}$, consider the case where $m=7$ and $b\equiv\modd{2} {4}$. Then, since $b=2k$, we have $k$ odd. If $n$ is even, then $(an+k)^2\equiv\modd{ k^{2}} {2}$ and so it follows that 
$$((an+k)^{2}-k^{2}+ac) \equiv k^2-7 \equiv -6 \equiv \modd{0} {2},$$
but
$$((an+k)^{2}-k^{2}+ac)\equiv k^2-7 \equiv -6 \equiv \modd{2} {4}.$$ 
Therefore $\nu_{2}(an^{2}+bn+c)=1$ when $n$ is even. Similarly, when $m=7$ and $b\equiv\modd{2} {4}$ if $n$ is odd, then $(an+k)^{2}$ is even. 
Therefore, 
$(an+k)^{2}-k^{2}+ac \equiv -7 \equiv \modd{1} {2}$.
Thus $\nu_{2}(an^{2}+bn+c)=0$ when $n$ is odd.

Now consider the case where $m=7$ and $b\equiv\modd{0}
{4}$. We have $b=2k$ with $k$ even. Thus, if $n$ is odd we have
$$(an+k)^{2}-k^{2}+ac \equiv -6\equiv \modd{0} {2} $$
and
$$(an+k)^{2}-k^{2}+ac \equiv k^2-7 \equiv -6 \equiv \modd{2} {4}.$$ 
Thus $\nu_{2}(an^{2}+bn+c)=1$ when
$n$ is odd. When $n$ is even we have $(an+k)^{2}-k^{2}+ac
\equiv -7 \equiv \modd{1} {2}$. Thus $\nu_{2}(an^{2}+bn+c)=0$ when $n$
is even.

The cases of $m\in\{2,3,6\}$ when $b\equiv\modd{0} {4}$ or $b\equiv\modd{2} {4}$ can be handled in the same fashion. For $m=5$, the valuations are slightly different.

Consider the case where $m=5$. Recall that $b=2k$ for some $k\in\mathbb{Z}$. Note that
$$
  b^2-4ac=4(5-8j),
$$
and hence $k^2-ac=5-8j$. Thus,
$$
  (an+k)^2-k^2+ac=(an+k)^2-5+8j.
$$
If $(an+k)$ is even, which is the case when both $n$ and $k$ are even or both $n$ and $k$ are odd, then $(an+k)^2-5+8j$ is odd.
Thus, $\nu_{2}(an^{2}+bn+c)=0$.

Now suppose that $(an+k)$ is odd, which is true when $n$ and $k$ have different parity. Then $(an+k)^2\equiv\modd{1} {4}$, and this implies
$$
  (an+k)^2-5+8j \equiv 1-5+8j \equiv -4+8j \equiv\modd{0} {4}.
$$
Thus, $\nu_{2}(an^{2}+bn+c)\geq 2$. \\
Since $(an+k)$ is odd, let $an+k=2d+1$, for some $d\in \mathbb{Z}$. Then,
\begin{align*}
 (an+k)^2-5+8j &= (2d+1)^2-5+8j\\
        &\equiv\modd{4(d^2+d-1)} {8}.
\end{align*}
Observe that $d^2+d-1$ is odd, regardless of whether $d$ is even or odd. Thus, $\nu_{2}(an^{2}+bn+c)<3$. Therefore,
$\nu_{2}(an^{2}+bn+c)=2$.
\end{proof}

\begin{lemma}\label{FinLem2}
Under the assumptions of Proposition~\ref{FinThm} (Case 3(c) of Theorem~\ref{MainThm1}) let $\ell\geq 2$ and suppose $0<i<\ell$.
At the $i^{th}$ level there is one terminal and one non-terminal node. Furthermore, the terminal node has valuation $2(i-1)$ and the non-terminal node has valuation at least $2i$.
\end{lemma}

First we need:
\begin{claim}\label{SuppLemm7}
Let $a,k\in\mathbb{Z}$ with $a$ odd. Let $g(n)=an+k$, then $(\nu_{2}(g(n)))_{n\geq0}$ creates an unbounded sequence.
\end{claim}
\begin{proof}
First note that the root of $ax+k=0$ is $x=-\frac{k}{a}$. Also note that $\nu_{2}(x)=\nu_{2}(-k)-\nu_{2}(a)$. Since $a$ is odd, $\nu_{2}(a)=0$. Therefore $\nu_{2}(x)=\nu_{2}(-k)\geq0$. By equation~\eqref{SuppLemma3.5}, $x\in\mathbb{Z}_{2}$, so Lemma~\ref{SuppLemm6} implies that $(\nu_{2}(g(n)))_{n\geq0}$ is an unbounded sequence.
\end{proof}

\begin{proof}
To prove Lemma~\ref{FinLem2}, we proceed by an inductive argument on $i$. Again, using the convention that $\Delta=m-8j$ where $j\in\mathbb{Z}$, for the base case $i=1$, note that
$b^{2}-4ac \equiv 4^{\ell}(m-8j) \equiv \modd{0} {4}$.
Recall that $b=2k$. First, assume that $k$ is even. If $n$ is even, then $an+k$ is even and so $(an+k)^{2}-k^{2}+ac \equiv \modd{0} {4}$. Thus $\nu_{2}(an^{2}+bn+c)=\nu_{2}((an+k)^{2}-k^{2}+ac)\geq2$ by Claim~\ref{SuppLemm7}. If $n$ is odd, then $(an+k)^{2}-k^{2}+ac \equiv \modd{1} {2}$,
and again using the technique of completing the square,
$\nu_{2}(an^{2}+bn+c)=0$. If $k$ is odd, a similar argument shows that $\nu_{2}(an^{2}+bn+c)\geq2$ when $n$ is odd. Observe also that Claim~\ref{SuppLemm7} can be used to show that $(\nu_2((an+k)^2))_{n\geq 0}$
forms an unbounded sequence therefore $\nu_{2}((an+k)^{2}-k^{2}+ac)\geq2$. Thus, the claim is true for $i=1$.

For the inductive step, notice that since $i<\ell$, 
it follows that $$b^{2}-4ac \equiv 4^{\ell}(m-8j) \equiv \modd{0} {2^{2i}}.$$
Suppose there exists an $i-1\geq 0$ such that $n=2^{i-1}q+r_{i-2}$ splits into two nodes: one node terminating with valuation $2(i-1)$ and the other node having valuation of at least $2i$.
We let $n=2^{i}q+r_{i-1}$ denote the non-terminating node,
where $r_{i-1}=\sum_{h=0}^{i-1}\alpha_{h}2^{h}$ with
$\alpha_{h}\in\left\{0,1\right\}$, for all $0\leq h\leq i-2$, and $q\in\mathbb{Z}$.
Then we have 
$$(an+k)^{2}-k^{2}+ac \equiv {(a(2^{i}q+r_{i-1})+k)^{2}} \equiv
\modd{0} {2^{2i}} ,$$
so $\nu_{2}(an^{2}+bn+c)\geq2i$. This also implies that 
$a(2^{i}q+r_{i-1})+k \equiv \modd {0} {2^{i}}$. Thus $ar_{i-1}+k=2^{i}\beta$ where $\beta\in\mathbb{Z}$. Now suppose that $k$ is even. (The proof for $k$ odd can be handled in the same fashion, and thus is omitted.) Since $k$ is even, then $r_{i-1}$ must be even.

Consider the $(i+1)^{st}$ level where $i+1<\ell$. Here again we
have 
$$b^{2}-4ac=\modd{4^{\ell}(m-8j)} {2^{2(i+1)}}\equiv0.$$
Moving to the next level, in the case $n=2^{i+1}q+r_{i-1}$ we have
\begin{align*}
    \nu_{2}((an+k)^{2}-4^{\ell-1}(m-8j))&=\nu_{2}((a(2^{i+1}q+r_{i-1})+k)^{2}-4^{\ell-1}(m-8j))\\
    &=\nu_{2}((2^{i+1}aq+ar_{i-1}+k)^{2}-4^{\ell-1}(m-8j))\\
    &=\nu_{2}((2^{i+1}aq+2^{i}\beta)^{2}-4^{\ell-1}(m-8j))\\
    &=\nu_{2}(2^{2i}(2aq+\beta)^{2}-2^{2(\ell-1)}(m-8j)),
\end{align*}
and in the case $n=2^{i+1}q+2^{i}+r_{i-1}$ we have
\begin{align*}
    &\nu_{2}((an+k)^{2}-4^{\ell-1}(m-8j))\\
		&\hspace{30pt}=\nu_{2}((a(2^{i+1}q+2^{i}+r_{i-1})+k)^{2}-4^{\ell-1}(m-8j))\\
    &\hspace{30pt}=\nu_{2}((2^{i+1}aq+2^{i}a+ar_{i-1}+k)^{2}-4^{\ell-1}(m-8j))\\
    &\hspace{30pt}=\nu_{2}((2^{i+1}aq+2^{i}a+2^{i}\beta)^{2}-4^{\ell-1}(m-8j))\\
    &\hspace{30pt}=\nu_{2}(2^{2i}(2aq+a+\beta)^{2}-2^{2(\ell-1)}(m-8j)).
\end{align*}
Since $\beta\in\mathbb{Z}$ either $2aq+\beta$ or $2aq+a+\beta$ is odd and the other is even. As long as $i+1<\ell$ then in the odd case the valuation is $2i$ and in the even case the valuation is at least $2(i+1)$.
\end{proof}

\begin{lemma}\label{FinLem3}
If $a$ is odd and $b$ is even with $b=2k$ for $k\in\mathbb{Z}$, and $b^{2}-4ac=4^{\ell}\Delta$ where $\ell\in\mathbb{Z}^+$ is as large as possible, $\Delta\equiv\modd{m} {8}$, and $m\in\{2,3,5,6,7\}$, then at the $\ell^{th}$ level the nodes of the 2-adic valuation tree terminate with valuations of $2(\ell-1)$, $2\ell-1$ or $2\ell$.

Suppose that $n=2^{\ell}q+r_{\ell-2}$.
If $an+k\equiv\modd{0} {2^{\ell}}$, then
\begin{displaymath}
    \nu_{2}(f(n))=\begin{cases}
         2(\ell-1),&\ \text{if}\ m=7,5,3;\\
         2\ell-1,&\ \text{if}\ m=6,2;
    \end{cases}
\end{displaymath}
and if $an+k\not\equiv\modd{0} {2^{\ell}}$, then
\begin{displaymath}
    \nu_{2}(f(n))=\begin{cases}
         2(\ell-1),&\ \text{if}\ m=6,2;\\
         2\ell-1,&\ \text{if}\ m=7,3;\\
         2\ell,&\ \text{if}\ m=5.
    \end{cases}
\end{displaymath}

Suppose that $n=2^{\ell}q+2^{\ell-1}+r_{\ell-2}$. If $an+k\equiv\modd{0} {2^{\ell}}$, then
\begin{displaymath}
    \nu_{2}(f(n))=\begin{cases}
         2(\ell-1),&\ \text{if}\ m=6,2;\\
         2\ell-1,&\ \text{if}\ m=7,3;\\
         2\ell,&\ \text{if}\ m=5;
    \end{cases}
\end{displaymath}
and if $an+k\not\equiv\modd{0} {2^{\ell}}$, then
\begin{displaymath}
    \nu_{2}(f(n))=\begin{cases}
         2(\ell-1),&\ \text{if}\ m=7,5,3;\\
         2\ell-1,&\ \text{if}\ m=6,2.
    \end{cases}
\end{displaymath}
\end{lemma}

\begin{proof}

By Lemma~\ref{FinLem2} there exists a non-terminating node $n=2^{\ell-1}q+r_{\ell-2}$ with $q\in\mathbb{Z}$ and $$\nu_{2}((an+k)^{2}-k^{2}+ac)\geq2(\ell-1).$$
Consider $n=2^{\ell}q+r_{\ell-2}$ with $q\in\mathbb{Z}$. By the same argument as in Lemma~\ref{FinLem2} and using the convention that $\Delta=m-8j$ where $j\in\mathbb{Z}$, we have $$(an+k)^{2}-k^{2}+ac=(2^{\ell}aq+2^{\ell-1}\beta)^{2}-2^{2(\ell-1)}(m-8j)=2^{2(\ell-1)}((2aq+\beta)^{2}+8j-m),$$ where $\beta\in\mathbb{Z}$. Recall that $a$ is odd. Then depending on whether $\beta$ is even or odd, simple calculations show the first two results.

In the case when $n=2^{\ell}q+2^{\ell-1}+r_{\ell-2}$ with $q\in\mathbb{Z}$ we have $$(an+k)^{2}-k^{2}+ac=2^{2(\ell-1)}((2aq+a+\beta)^{2}+8j-m),$$ where $\beta\in\mathbb{Z}$.
Then again depending on whether $\beta$ is odd or even, it is straightforward to show the last two results.
\end{proof}

\section{Structure of finite trees}\label{SecStrucTree}

The section describes the overall structure of finite trees, continuing the discussion of Case 3(c) of Theorem~\ref{MainThm1}, in which $a$ is odd, $b$ is even, $b^{2}-4ac=4^{\ell}\Delta$ where $\Delta\equiv\modd{m} {8}$, and $m\in\{2,3,5,6,7\}$. Throughout this section, we make use of several operators. The operators allow us to track changes from very easily described trees, which we call type $(\ell,1)$, to more complicated trees.

\begin{definition}[Translation operator,~\cite{Grafakos}]
For quadratics of the form $f(n)=an^{2}+bn+c$ we define $\tau^{s}(f)(n)=f(n-s)$ for $s\in\mathbb{R}$, namely $\tau^{s}(f)(n)=a(n-s)^{2}+b(n-s)+c=an^{2}+(b-2as)n+(c+as^{2}-bs)$.
\end{definition}
\begin{proposition}\label{FinPropStruc1}
Let the assumptions of Proposition~\ref{FinThm} hold for the function $f(n)=an^{2}+bn+c$ and suppose $s\in\mathbb{Z}$. 
Then we have the following relationship
\begin{equation*}
\nu_{2}(f(2^{i}q+r_{i-1}))
=\nu_{2}(\tau^{s}f(2^{i}q+(r_{i-1}+s) \bmod {2^{i}})).
\end{equation*}
That is the valuations $\nu_{2}(f(n))$ at the node of the form $n=2^{i}q+r_{i-1}$ are moved to the node of the form $n=\modd{2^{i}q+(r_{i-1}+s)} {2^{i}}$ under the operation $\tau^{s}$. \end{proposition}

\begin{proof}
Note that finite trees with $\ell$ levels correspond to periodic sequences
with a period equal to $2^{\ell}$. Since $\tau^s$ is a translation operator, every element in the sequence $(\nu_2(f(n)))_{n\geq0}$ is moved over $s$ spaces.
\end{proof}

\begin{definition}[$S$-operator]\label{defn-S-operator}
Let $a$ be a positive, odd integer. For quadratics of the form $f(n)=n^{2}+bn+ac$ we define $S^{a}(f)(n)=an^{2}+bn+c$. Likewise, for quadratics of the form $f(n)=an^{2}+bn+c$ define $S^{a^{-1}}(f)(n)=n^{2}+bn+ac$.
\end{definition}

In general, the $S$-operator need not output a quadratic function with an integer constant term. However, the present work only applies $S^a$ to functions whose output has integer coefficients.

\begin{definition}[Dilation operator,~\cite{Grafakos}]
For quadratics of the form $f(n)=an^{2}+bn+c$ we define $\delta^{s}(f)(n)=f(sn)$ for $s\in\mathbb{R}$, namely $\delta^{s}(f)(n)=a(sn)^{2}+b(sn)+c$.
\end{definition}

\begin{lemma}\label{FinLemStruc2}
Under the assumptions of Proposition~\ref{FinThm} the trees created by $f(n)=n^{2}+bn+ac$ and $S^{a}(f)(n)$ where $a\in\mathbb{Z}$ have the same number of levels. Similarly, the trees created by $g(n)=an^{2}+bn+c$ and $\tau^{s}(g)(n)$ where $s\in\mathbb{Z}$ have the same number of levels.
\end{lemma}

\begin{proof}
The assumptions of Proposition~\ref{FinThm} represent Cases 3(b) and 3(c) of Theorem~\ref{MainThm1}. Simple calculations show that the discriminants of $f(n)$ and $S^{a}(f)(n)$ are the same, and that the discriminants of $g(n)$ and $\tau^{s}(g)(n)$ are the same. The conclusions then follow directly from Proposition~\ref{FinThm}.
\end{proof}

\begin{proposition}\label{FinPropStruc2}
Let the assumptions of Proposition~\ref{FinThm} hold and suppose $f(n)=n^{2}+bn+ac$. Then we have the following relationship
\begin{equation*}
\nu_{2}(f(2^{i}q+r_{i-1}))=\nu_{2}(S^{a}(f(2^{i}q+a^{-1}\cdot r_{i-1}))).
\end{equation*}
That is the valuation $\nu_{2}(f(n))$ at the node the form $n=2^{i}q+r_{i-1}$ is moved to the node of the form of $n=2^{i}q+(a^{-1}\cdot \modd{r_{i-1})} {2^{i}}$ under the operation $S^{a}$. In this context $a^{-1}$ is the inverse of $\modd{a} {2^{i}}$.
\end{proposition}

\begin{proof}
Since $a$ is odd, note that
$$\nu_{2}(S^{a}(f)(n))=\nu_{2}((an^{2}+bn+c))=\nu_{2}((an)^{2}+b(an)+ac)=\nu_{2}(\delta^{a}(f)(n)),$$ 
where $\delta^{a}(f)(n)=f(an)$ is the dilation operator. Thus, the valuation of $f(n)$ for $n=2^{i}q+r_{i-1}$ is the same as the valuation of $n'=2^{i}(a^{-1}q)+a^{-1}\cdot r_{i-1}$ after the $S^a$-operator is applied.
\end{proof}

Suppose that $f(n)=an^{2}+bn+c$ creates a finite tree. We say that this tree is \textit{type $(\ell,1)$}, for $\ell\geq2$, if at every level the non-terminating node is of the form $n=2q$ or $n=2^{i}q+2^{i-2}+\cdots+2^{1}+2^{0}$ for $i<\ell$ and the tree has $\ell$ levels. We also say that a quadratic function is type $(\ell,1)$ if it creates an $(\ell,1)$ tree. That is, $f(n)$ creates a finite tree of the following form:

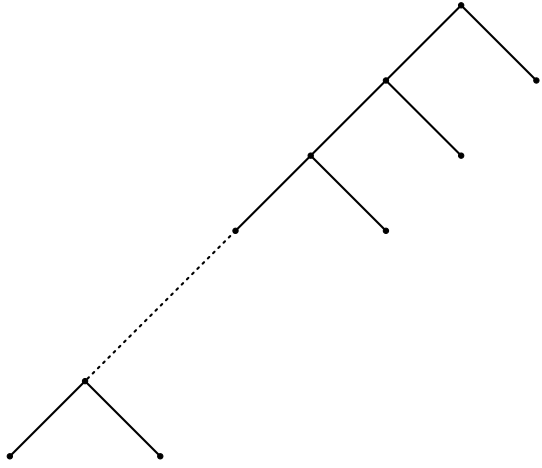
\begin{figure}[h!]
    \centering
\begin{equation*}
\begin{tikzpicture}[line cap=round,line join=round,>=triangle 45,x=1cm,y=1cm]
\draw [line width=0.8pt] (0,7)-- (-1,6);
\draw [line width=0.8pt] (0,7)-- (1,6);
\draw [fill=black] (0,7) circle (1pt);
\draw [fill=black] (-1,6) circle (1pt);
\draw [fill=black] (1,6) circle (1pt);
\draw [line width=0.8pt] (-1,6)-- (-2,5);
\draw [line width=0.8pt] (-1,6)-- (0,5);
\draw [fill=black] (0,5) circle (1pt);
\draw [fill=black] (-2,5) circle (1pt);
\draw [line width=0.8pt] (-2,5)-- (-3,4);
\draw [line width=0.8pt] (-2,5)-- (-1,4);
\draw [fill=black] (-3,4) circle (1pt);
\draw [fill=black] (-1,4) circle (1pt);
\draw [line width=0.8pt,dotted] (-3,4)-- (-5,2);
\draw [line width=0.8pt] (-5,2)-- (-6,1);
\draw [line width=0.8pt] (-5,2)-- (-4,1);
\draw [fill=black] (-5,2) circle (1pt);
\draw [fill=black] (-6,1) circle (1pt);
\draw [fill=black] (-4,1) circle (1pt);
\end{tikzpicture}
\end{equation*}
    \caption{The form of trees of type $(\ell,1)$.}
    \label{fig:ell_1_tree}
\end{figure}
\vspace{0.2in}

 Here, we suppose that $\ell\geq2$ because $\ell=1$ creates a tree with one level, see Lemma~\ref{FinLem1}, and the directional behavior we seek to classify is not defined. The conditions $4a^2-4ac=4^{\ell}\Delta$ for $\ell\in\mathbb{Z}^+$ as large as possible, $\Delta\equiv\modd{m} {8}$, and $m\in\{2,3,5,6,7\}$ imply $c$ must be odd.

\begin{proposition}\label{FinPropStruc3}
Under the assumptions of Proposition~\ref{FinThm}, if $c$ is odd and $\ell\geq2$ is an integer, then a quadratic of the form $f(n)=an^{2}+2an+c$ creates a tree that is of type $(\ell,1)$. Furthermore, we have
\begin{displaymath}
\nu_{2}(f(n))=\begin{cases}
0,&\ \text{if}\ n\equiv\modd{0} {2};\\
2(i-1),&\ \text{if}\ n\equiv\modd{\sum_{k=0}^{i-2}2^{k}} {2^{i}}\ \text{with}\ 2\leq i<\ell;\\
2(\ell-1),&\ \text{if}\ n\equiv\modd{\sum_{k=0}^{\ell-2}2^{k}} {2^{\ell}}\ \text{and}\ m=6,2;\\
2\ell-1,&\ \text{if}\ n\equiv\modd{\sum_{k=0}^{\ell-2}2^{k}} {2^{\ell}}\ \text{and}\ m=7,3;\\
2\ell,&\ \text{if}\ n\equiv\modd{\sum_{k=0}^{\ell-2}2^{k}} {2^{\ell}}\ \text{and}\ m=5;\\
2\ell-1,&\ \text{if}\ n\equiv\modd{\sum_{k=0}^{\ell-1}2^{k}} {2^{\ell}}\ \text{and}\ m=6,2;\\
2(\ell-1),&\ \text{if}\ n\equiv\modd{\sum_{k=0}^{\ell-1}2^{k}} {2^{\ell}}\ \text{and}\ m=7,5,3.
\end{cases}
\end{displaymath}
\end{proposition}
\begin{proof}
In light of Lemma~\ref{FinLem2}, we know that if a node is non-terminating, then it produces two nodes that either both terminate (i.e., these nodes are at the $\ell^{th}$ level) or one node is non-terminating and the other is terminating. So in order to show that the tree is of type $(\ell,1)$, we only need to confirm that nodes corresponding to $n=2^{i}q+2^{i-1}+\cdots+2^{1}+2^{0}$, where $1\leq i\leq\ell$, are always non-terminating. Since $a$ is odd, completing the square and using the convention that $4a^{2}-4ac=4^{\ell}\Delta$ where $\Delta=m-8j$ where $j\in\mathbb{Z}$ gives
\begin{align*}
\nu_{2}(f(n))&=\nu_{2}(an^{2}+2an+c)=\nu_{2}(a(n+1)^{2}-a+c)\\
&=\nu_{2}(a^{2}(n+1)^{2}-a^{2}+ac)=\nu_{2}(a^{2}(n+1)^{2}-4^{\ell-1}(m-8j))\\
&=\nu_{2}(a^{2}(2^{i}q+2^{i-1}+2^{i-2}+\cdots+2+1+1)^{2}-4^{\ell-1}(m-8j))\\
&=\nu_{2}(a^{2}(2^{i}q+2^{i})^{2}-4^{\ell-1}(m-8j))\\
&=\nu_{2}(a^{2}4^{i}(q+1)^{2}-4^{\ell-1}(m-8j)).
\end{align*}
If $q$ is odd, then $n=2^{i}q+2^{i-1}+\cdots+2^{1}+2^{0}$ is the non-terminating node, provided $i<\ell$, and produces two nodes one of which does not terminate. If $i=\ell$, then both nodes terminate by Proposition~\ref{FinThm}.

The nodes that terminate are of the form $n=2^{i}q+2^{i-2}+\cdots+2^{1}+2^{0}$ when $1\leq i<\ell$. The case when $n=2q$ is handled by the proof of Lemma~\ref{FinLemStruc2}. For the case $1<i<\ell$, by the same calculation as above we have $$\nu_{2}(f(n))=\nu_{2}(a^{2}2^{2(i-1)}(2q+1)^{2}-4^{\ell-1}(m-8j))$$ Since $2q+1$ is odd and $i<\ell$ the valuation must be $2(i-1)$.

In the case when $n=2^{\ell}q+2^{\ell-2}+\cdots+2^{1}+2^{0}$ we have $$\nu_{2}(f(n))=\nu_{2}(a^{2}2^{2(\ell-1)}(2q+1)^{2}-4^{\ell-1}(m-8j))$$ Thus the valuation must be $2(\ell-1)$ if $m=6,2$, or $2\ell-1$ if $m=7,3$ or $2\ell$ if $m=5$.

Finally if $n=2^{\ell}q+2^{\ell-2}+\cdots+2^{1}+2^{0}$ we have $\nu_{2}(f(n))=\nu_{2}(a^{2}2^{2\ell}(q+1)^{2}-4^{\ell-1}(m-8j))$. Thus the valuation must be $2(\ell-1)$ if $m=7,5,3$ or $2\ell-1$ if $m=6,2$.
\end{proof}

If the function $f(n)=an^{2}+bn+c$ meets the assumptions of Proposition~\ref{FinThm} (Case 3(c) of Theorem~\ref{MainThm1}) note if we define the function $g(n)=n^{2}+2n-\left(1-\frac{b}{2}\right)^{2}+2\left(1-\frac{b}{2}\right)+ac$, then 
it follows that $S^{a}(\tau^{1-\frac{b}{2}}(g))(n)=f(n)$. Therefore, by Propositions~\ref{FinPropStruc1},~\ref{FinPropStruc2}, and~\ref{FinPropStruc3} we immediately have the following corollary.

\begin{corollary}\label{FinCor}
If $f(n)=an^{2}+bn+c$ meets the assumptions of Proposition~\ref{FinThm} (Case 3(c)) with $\ell\geq 2$, then 
\begin{displaymath}
\nu_{2}(f(n))=\begin{cases}
0,&\ \text{if}\ n\equiv\modd{ a^{-1}\left(1-\frac{b}{2}\right)} {2};\\
2(i-1),&\ \text{if}\ n\equiv\modd{ a^{-1}\left(2^{i-1}-\frac{b}{2}\right)} {2^{i}}\ \text{with}\ 2\leq i<\ell;\\
2(\ell-1),&\ \text{if}\ n\equiv\modd{ a^{-1}\left(2^{\ell-1}-\frac{b}{2}\right)} {2^{\ell}} \text{and}\ m=6,2;\\
2\ell-1,&\ \text{if}\ n\equiv\modd{ a^{-1}\left(2^{\ell-1}-\frac{b}{2}\right)} {2^{\ell}}\ \text{and}\ m=7,3;\\
2\ell,&\ \text{if}\ n\equiv\modd{ a^{-1}\left(2^{\ell-1}-\frac{b}{2}\right)} {2^{\ell}}\ \text{and}\ m=5;\\
2\ell-1,&\ \text{if}\ n\equiv\modd{ a^{-1}\left(2^{\ell}-\frac{b}{2}\right)} {2^{\ell}}\ \text{and}\ m=6,2;\\
2(\ell-1),&\ \text{if}\ n\equiv\modd{ a^{-1}\left(2^{\ell}-\frac{b}{2}\right)} {2^{\ell}}\ \text{and}\ m=7,5,3;\\
\end{cases}
\end{displaymath}
where $a^{-1}$ is the inverse of $\modd{a} {2^{\ell}}$.
\end{corollary}

\begin{proof}
Simply note that $g$ is type $(\ell,1)$ and recall the ways in which the operators affect the function $g$. Each terminating node, under the operators, moves from $n=2^{i}q+r_{i-2}$ to $n=2^{i}q+\modd{a^{-1}\left(r_{i-2}+1-\frac{b}{2}\right)} {2^{i}}$. In the case of type $(\ell,1)$ we have $r_{i-2}=\sum_{k=0}^{i-2}2^{k}$. Thus $r_{i-2}+1=2^{i-1}$ in each case.
\end{proof}

\section{Acknowledgments}
The authors would like to thank Dr.\ Victor Moll for suggesting this topic. We would also like to thank following institutions for providing support to collaborate: ICERM, AIM, Kentucky Wesleyan College, Ursinus College, and Stephen F. Austin State University. We are grateful to our other colleagues for their ongoing support: Dr.\ Maila Brucal-Hallare, Dr.\ Jean-Claude Pedjeu, and Dr.\ Bianca Thompson. And, finally, we are very grateful to the reviewer who provided many very helpful and insightful comments.

\section*{Appendix: figures illustrating trees and tables of values for $2$-adic valuation sequences of some quadratic functions}
In the following tree representations, a closed circle indicates a terminating node and an open circle indicates a non-terminating node.
\begin{figure}[H]
    \centering
    \includegraphics[scale=0.4]{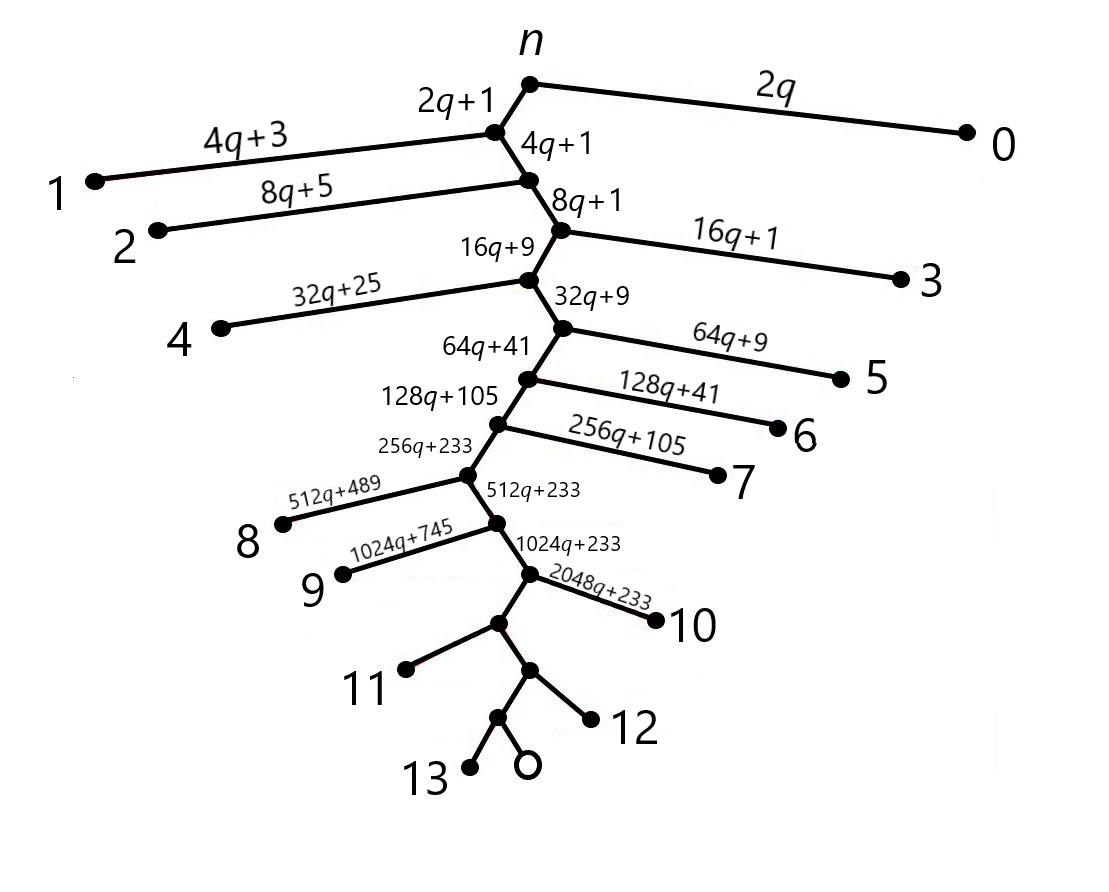}
    \vspace{0.2in}
    \small
    \begin{tabular}{c|cccc ccccc ccc}
$n$&0&1&2&3&4&5&6&7&8&9&10&11\\ \hline
$f_1(n)$&$-25$&$-8$&17&50&91&140&197&262&335&416&505&602\\
$\nu_{2}(f_12(n))$&0&3&0&1&0&2&0&1&0&5&0&1
\end{tabular}    
\vspace{0.2in}
\begin{tabular}{c|cccc ccccc cc}
    $n$&12&13&14&15&16&17&18&19&20&21\\ \hline
    $f_1(n)$&707&820&941&1070&1207&1352&1505&1666&1835&2012\\
    $\nu_{2}(f_1(n))$&0&2&0&1&0&3&0&1&0&2
    \end{tabular}
    \vspace{0.2in}
    \caption{The 2-adic valuation tree for $f_1(n)=4n^{2}+13n-25$. Theorem~\ref{MainThm1} predicts that $(\nu_2(f_1(n)))_{n\geq 0}$ is an unbounded sequence, as it satisfies Case 2.}
    \label{fig:Ex8}
\end{figure}

\begin{figure}[H]
    \centering
    \includegraphics[scale=0.4]{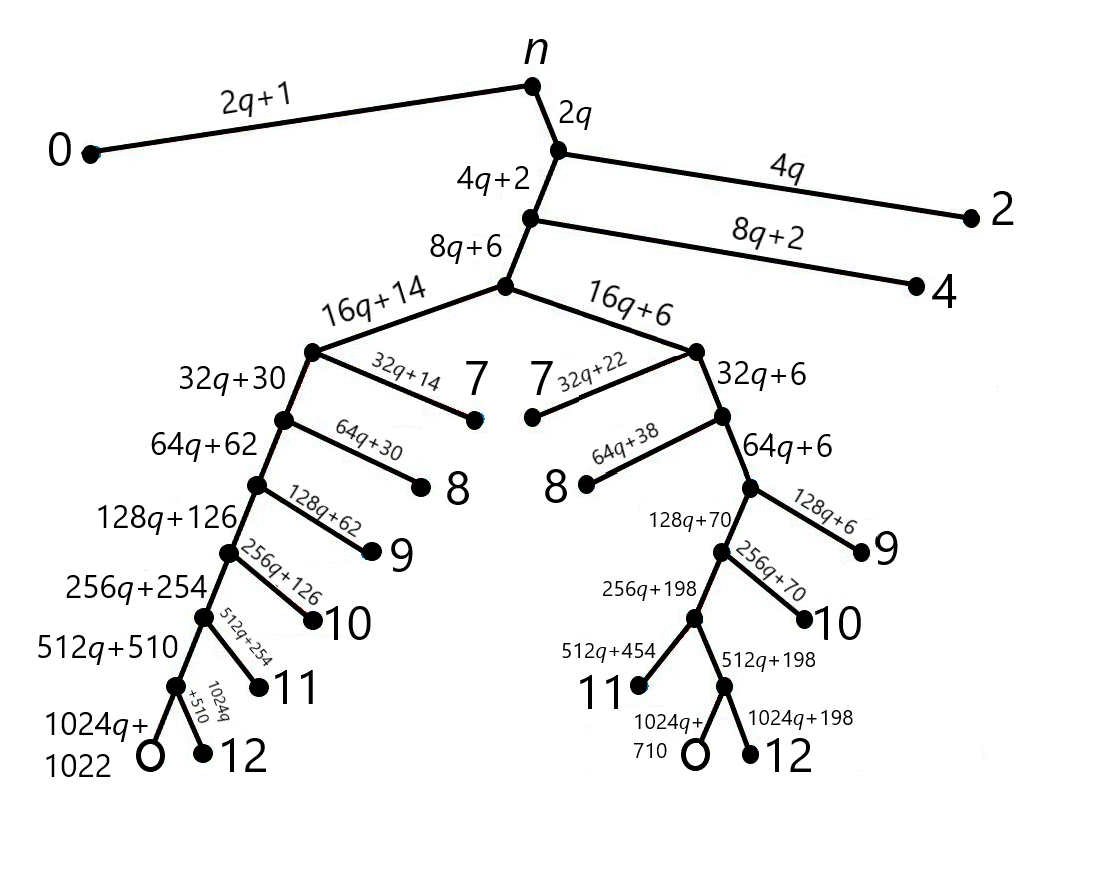}
    \vspace{0.2in}
    \small
    \begin{tabular}{c|cccc ccccc ccc}
$n$&0&1&2&3&4&5&6&7&8&9&10\\ \hline
$f_2(n)$&$-28$&$-3$&48&125&228&357&512&693&900&1133&1392\\
$\nu_{2}(f_2(n))$&2&0&4&0&2&0&9&0&2&0&4&\\
\end{tabular}
\vspace{0.2in}
\begin{tabular}{c|cccc ccccc ccc}
    $n$&11&12&13&14&15&16&17&18&19\\ \hline
    $f_2(n)$&1677&1988&2325&2688&3077&3492&3933&4400&4893\\
    $\nu_{2}(f_2(n))$&0&2&0&7&0&2&0&4&0\\
    \end{tabular}
    \vspace{0.2in}
    \caption{The 2-adic valuation tree and data for $f_2(n)=13n^{2}+12n-28$. Notice that Theorem~\ref{MainThm1} predicts that $(\nu_2(f_2(n)))_{n\geq 0}$ is an unbounded sequence, as it satisfies Case 3(a) since $12^{2}-4\cdot13(-28)=4^{3}(1-8(-3))$.}
    \label{fig:Ex7}
\end{figure}

\begin{figure}[H]
    \centering
    \includegraphics[scale=0.4]{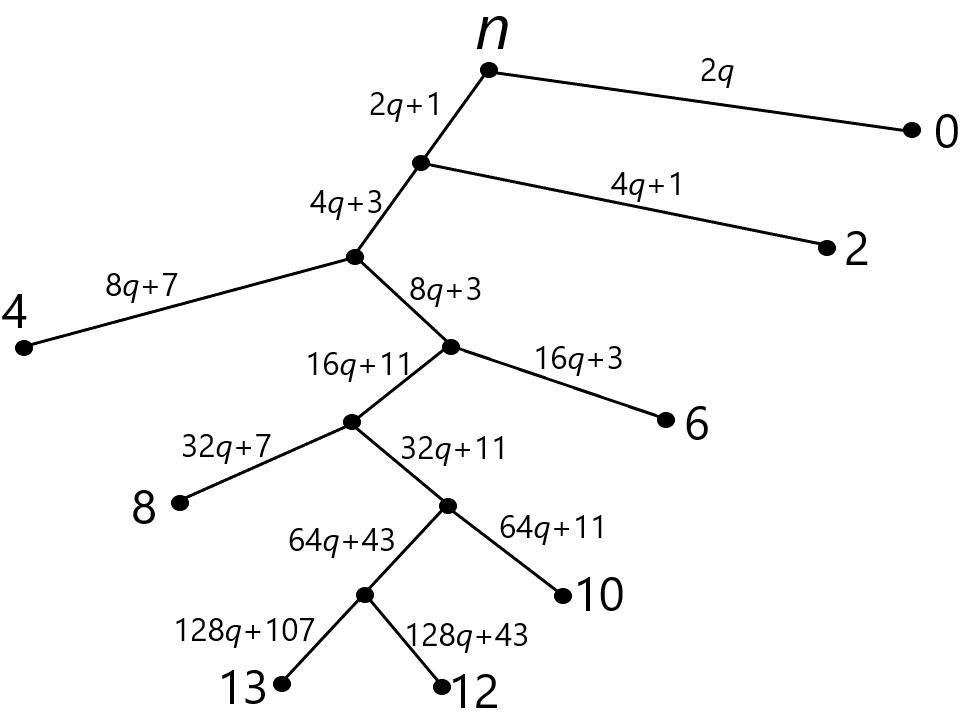}
    \vspace{0.2in}
    \small
\begin{tabular}{c|cccc ccccc}
$n$&0&1&2&3&4&5&6&7\\\hline
$f_3(n)$&25559&26716&27903&29120&30367&31644&32951&34288\\
$\nu_{2}(f_3(n))$&0&2&0&6&0&2&0&4
\end{tabular}    
\vspace{0.2in}
\begin{tabular}{c|cccc ccccc}
    $n$&8&9&10&11&12&13&14&15\\\hline
    $f_3(n)$&35655&37052&38479&39936&41423&42940&44487&46064\\
    $\nu_{2}(f_3(n))$&0&2&0&10&0&2&0&4
    \end{tabular}
    \vspace{0.2in}
    \caption{The 2-adic valuation tree and data for $f_3(n)=15n^2+1142n+25559$. Notice that Theorem~\ref{MainThm1} predicts that $(\nu_{2}(f_3(n))_{n\geq0}$ is a bounded sequence, as it satisfies Case 3(c) since $1142^2-4\cdot 15 \cdot 25559=4^7(2-8\cdot 2)$.}
    \label{fig:Ex5}
\end{figure}

\begin{figure}[H]
    \centering
    \includegraphics[scale=0.4]{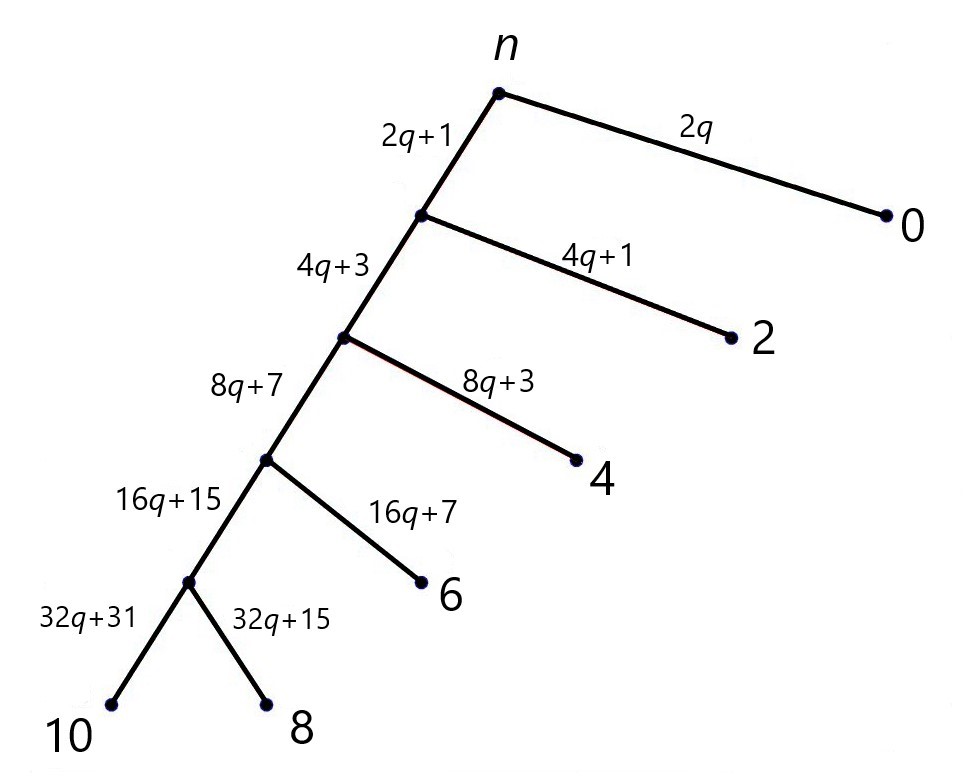}
    \vspace{0.2in}
    \small
    \begin{tabular}{c|cccc ccccc cc}
$n$&0&1&2&3&4&5&6&7&8&9\\\hline
$f_4(n)$&1125&1236&1357&1488&1629&1780&1941&2112&2293&2484\\
$\nu_{2}(f_4(n))$&0&2&0&4&0&2&0&6&0&2
\end{tabular}
\vspace{0.2in}
\begin{tabular}{c|cccc ccccc cc}
    $n$&10&11&12&13&14&15&16&17&18&19\\\hline
    $f_4(n)$&2685&2896&3117&3348&3589&3840&4101&4372&4653&4944\\
    $\nu_{2}(f_4(n))$&0&4&0&2&0&8&0&2&0&4
    \end{tabular}
    \vspace{0.2in}
    \caption{The 2-adic valuation tree and data for $f_4(n)=5n^{2}+106n+1125$.  Notice that Theorem~\ref{MainThm1} predicts that $(\nu_{2}(f_4(n))_{n\geq0}$ is a bounded sequence, as it satisfies Case 3(c) since $106^2-4\cdot 5 \cdot 1125=4^5(5-8\cdot 2)$.}
    \label{fig:Ex6}
\end{figure}

\end{document}